\documentclass[11pt]{amsart}
\usepackage{amssymb}
\usepackage{mathrsfs}
\usepackage{enumerate}
\makeatletter
\@namedef{subjclassname@2010}{%
  \textup{2010} Mathematics Subject Classification}
\makeatother

\newtheorem{theorem}{Theorem}[section]
\newtheorem{proposition}{Proposition}[section]
\newtheorem{corollary}{Corollary}[section]
\newtheorem{definition}{Definition}[section]
\newtheorem{remark}{Remark}[section]

\begin{document}

\title[Robustness of Controlled Fusion Frames under Erasures]{Robustness of Controlled Fusion Frames under Erasures of Subspaces and the Approximation Operator}

\author[R. Ahmadi]{Reza Ahmadi$^*$}
\address{Institute of Fundamental Sciences\\University of Tabriz\\, Iran\\}
\email{rahmadi@tabrizu.ac.ir}

\author[Gh. Rahimlou]{Gholamreza Rahimlou}
\address{Department of Mathematics, Technical and Vocational University (TVU), Tehran,
, Iran\\}
\email{ghrahimlo@tvu.ac.ir}

\author[V. Sadri]{Vahid Sadri}
\address{Department of Mathematics,\\Technical and Vocational University (TVU), Tehran, Iran.}
\email{vahidsadri57@gmail.com}

\subjclass[2010]{Primary 42C15; Secondary 46C99, 41A58}
\keywords{Controlled frame, controlled-fusion frame, trace-class operator, Erasures.\\
*Corresponding Author}
\maketitle

\begin{abstract}
Controlled frames which presented  to improve the numerical output of iterative algorithms for inverting the frame operator,  have been  introduced by Balazs and et al.  Also, these frames are used by Bogdanova and et al. for spherical wavelets.  In this note, we aim to study some new results about controlled fusion frames and behavior of them under erasures. Finally, we will introduce the approximation operator for these frames and a result of the composition of  the operators of two controlled fusion frames will be proved.
\end{abstract}

\section{Introduction}
Frames which are an expansion of bases in Hilbert spaces, were first proposed
by Duffin and Schaeffer to deal with  nonharmonic Fourier series in \cite{Duffin}. Indeed, frames are sets of vectors in a Hilbert space that yield one representation for each vector in the space, but which  may have many different representations for a given vector.
 They play a fundamental role in mathematics, filter bank theory, coding and communications, signal processing, system modeling and model sensor networks (e. g. \cite{Musazadeh, Strohmer, Feichtinger, cassaza}) and they have some generalizations as fusion frames, g-frames, c-frames, $K$-frames, controlled frame, woven and etc.

The notation of fusion frames have been introduced by Casazza and Kutyniok in \cite{cas1}, provide a means to improve robustness or develop feasible reconstruction algorithms. They were able to construct robustness of these frames under erasures in \cite{cas2} by using the methods suggested by Bodman \cite{bod} for Parseval frames under the term weighted projective resolution of the identity for optimal transmission of quantum states and for packet encoding.

Recently, controlled frames have been presented by Balazs and et al. in \cite{Balazs} to improve the numerical efficiency of interactive algorithms for inverting the frame operator on Hilbert spaces. After, controlled frames have been studied for another kind of frames in \cite{Hua, Musazadeh, Reza, Rahimi}.

This manuscript is organized as follows: In Section 2, we study a new identity for the eigenvalues of the controlled frame operator and review the notation of controlled fusion frames with some results about these operators. In Section 3, we focus on the study of robustness of controlled fusion frames under erasures. Finally, we will present an interesting result about the composition of  the operators of two controlled fusion frames.

Throughout this paper, $H$ and $K$ are separable Hilbert spaces, $\mathcal{B}(H,K)$ is the family of the all bounded linear operators on $H$ into $K$. When $H=K$ we use $\mathcal{B}(H)$ instead of $\mathcal{B}(H,H)$. We denote $\mathcal{GL}(H)$ as the set of all bounded linear operators which have bounded inverses and $\mathcal{GL}^{+}(H)$ is the set of all positive operators in $\mathcal{GL}(H)$.
Also, we define $\pi_W$ as the orthogonal projection from $H$ onto a closed subspace $V\subseteq H$.
%%%%%%%%%%%%%%%%%%%%%%%%%%%%%%%%%%
\section{Review of Controlled  Frames}
In this section, we review the notations of controlled frames, fusion frames and present a useful identity about the eigenvalues of the controlled  frame operator. Throughout this paper, $C, C'\in\mathcal{GL}(H)$  and  $\Bbb I\subset\Bbb Z$.
\begin{definition}
Let $\lbrace f_{i}\rbrace_ {i\in\Bbb I}$ be  members of $H$.  We call $\lbrace f_{i}\rbrace_ {i\in\Bbb I}$  a $(C,C')$-controlled frame  for $H$ if there exist constants $0< A \leq B< \infty$ such that for each $f \in H$,
\begin{align}\label{con1}
A \Vert f \Vert^{2} \leq \sum _{i\in\Bbb I}\langle C'f,f_i\rangle\langle f_i, Cf\rangle\leq B \Vert f\Vert^{2}.
\end{align}
\end{definition}
In particular, we say $\lbrace f_{i}\rbrace_ {i\in\Bbb I}$  a $C$-controlled frame for $H$ if there exist constants $0< A \leq B< \infty$ such that for each $f \in H$,
\begin{align}
A \Vert f \Vert^{2} \leq \sum _{i\in\Bbb I}\langle f,f_i\rangle\langle Cf_i, f\rangle\leq B \Vert f\Vert^{2}.
\end{align}
When $A=B=1$, we say $\lbrace f_{i}\rbrace_ {i\in\Bbb I}$ a Parseval $(C,C')$-controlled frame and if the right hand of \eqref{con1} holds, $\lbrace f_{i}\rbrace_ {i\in\Bbb I}$ is called a $(C,C')$-controlled Bessel sequence. In this case, the controlled  frame operator is defined by
$$S_{CC'}f=\sum_{i\in\Bbb I}\langle C'f, f_i\rangle C^*f_i,\qquad f\in H.$$
The next result gives a new identity about the eigenvalues of $S_{CC'}$.
\begin{proposition}\label{eigin1}
Let $\lbrace f_{i}\rbrace_ {i\in\Bbb I}$ be a  $(C,C')$-controlled frame for $H$ where $\dim H=n$ and $\{\lambda_i\}_{i=1}^{n}$ be the eigenvalues of the operator $S_{CC'}$. Then
$$\sum_{i=1}^{n}\lambda_i=\sum_{i\in\Bbb I}\langle C^*f_i, C'^*f_i\rangle.$$
\end{proposition}
\begin{proof}
Suppose that  $\{e_i\}_{i\in\Bbb I}$ is an orthonormal basis for $H$, then
\begin{align*}
\sum_{i=1}^{n}\lambda_i&=\sum_{i=1}^{n}\lambda_i\Vert e_i\Vert^2\\
&=\sum_{i=1}^{n}\langle S_{CC'}e_i, e_i\rangle\\
&=\sum_{i=1}^{n}\sum_{j\in\Bbb I}\big\langle\langle C'e_i, f_j\rangle C^*f_j, e_i\big\rangle\\
&=\sum_{j\in\Bbb I}\sum_{i=1}^{n}\langle e_i,C'^*f_j\rangle\langle C^*f_j,e_i\rangle\\
&=\sum_{j\in\Bbb I}\langle C^*f_j, C'^*f_j\rangle.
\end{align*}
\end{proof}
\begin{corollary}\label{eigin2}
Let $\lbrace f_{i}\rbrace_ {i\in\Bbb I}$ be a Parseval $(C,C')$-controlled frame for $H$ where $\dim H=n$. Then
$$\sum_{i\in\Bbb I}\langle C^*f_i, C'^*f_i\rangle=n.$$
\end{corollary}
Next, we focus on controlled fusion frames for a Hilbert space which have been introduced by Khosravi and Musazadeh in \cite{Khosravi}.
\begin{definition}
Let $\lbrace W_{i}\rbrace_ {i\in\Bbb I}$ be a collection of closed subspace in $H$ and  $\lbrace v_{i}\rbrace_ {i\in\Bbb I}$ be a family of weights, i.e. $v_{i}>0$,  $i \in\Bbb I$. The sequence  $W=\lbrace (W_{i},v_{i})\rbrace _ {i \in \Bbb I}$ is called a  $(C,C')$-controlled fusion frame  for $H$ if there exist constants $0< A \leq B< \infty$ such that for all $f \in H$
\begin{align*}
A \Vert f \Vert^{2} \leq \sum _{i\in\Bbb I} v_{i}^{2} \langle \pi_{W_{i}} C^{\prime}f,\pi_{W_{i}} Cf \rangle  \leq B \Vert f\Vert^{2}.
\end{align*}
\end{definition}
Throughout this paper, $W=\lbrace (W_{i},v_{i})\rbrace _ {i \in I}$ unless otherwise stated. We call $W$ a tight $(C,C')$-controlled fusion frame, if $A=B$ and a Parseval $(C,C')$-controlled fusion frame provided $A=B=1$. We call $W$ is a $C^2$-controlled fusion frame  if $C=C'$.
If only the second inequality is required, We call $(C,C')$-controlled Bessel fusion sequence  with bound $B$.
If $W$ is a $(C,C')$-controlled fusion frame and $C^*\pi_{W_i}C'$ is a positive operator for each
$i\in\Bbb I$, then $C^*\pi_{W_i}C'=C'^*\pi_{W_i}C$ and we have
\begin{align*}
A \Vert f \Vert^{2} \leq \sum _{i\in\Bbb I} v_{i}^{2} \Vert(C^* \pi_{W_{i}} C')^{\frac{1}{2}}f\Vert^2 \leq B \Vert f\Vert^{2}.
\end{align*}
We define the controlled analysis operator by (for more details, we refer to \cite{Khosravi})
\begin{align*}
T_{W}&:H \rightarrow \mathcal{K}_{2,W} \\
T_{W}f=&\lbrace v_{i}(C^{*}\pi_{W_{i}}C')^{\frac{1}{2}}f\rbrace_{i\in\Bbb I},
\end{align*}
where
\begin{align*}
\mathcal{K}_{2,W}:=\big\lbrace \lbrace v_i(C^{*}\pi_{W_{i}} C')^{\frac{1}{2}}f\rbrace_{i\in\Bbb I} \ : \ f\in H\big\rbrace \subset (\bigoplus_{i \in\Bbb I} H)_{l^{2}}.
\end{align*}
It is easy to see that $\mathcal{K}_{2,W}$ is closed and $T_{W}$ is well defined. Moreover $T_{W}$ is a bounded linear operator with the adjoint operator $T ^{*}_{W}$ (the controlled synthesis operator) defined by
\begin{align*}
T^{*}_{W}&:\mathcal{K}_{2,W} \rightarrow H \\
T ^{*}_{W}\lbrace v_i(C^{*}&\pi_{W_{i}} C')^{\frac{1}{2}}f\rbrace_{i\in\Bbb I}=\sum _{i\in\Bbb I} v_{i}^{2}C^{*}\pi_{W_{i}} C'f.
\end{align*}
Therefore, we define the controlled fusion frame operator $S_{W}$ on $H$ by
\begin{align*}
S_{W}f=T ^{*}_{W}T_{W}(f)=\sum _{i\in\Bbb I} v_{i}^{2}C^{*}\pi_{W_{i}} C'f.
\end{align*}
It is easy to check that $S_W$ is a well defined and
$$AId_H\leq S_W\leq B Id_H,$$
hence, $S_W$ is a bounded, invertible, self-adjoint and positive linear operator. The next result provides some conditions to get a controlled fusion frame.
\begin{theorem}\label{char1}
Let $C^*\pi_{W_i}C'$ is a positive operator for each
$i\in\Bbb I$. A sequence  $W=\lbrace (W_{i},v_{i})\rbrace _ {i \in \Bbb I}$ is  a  $(C,C')$-controlled fusion frame  for $H$ with bounds $A$ and $B$ if and only if the operator
\begin{align*}
T^{*}_{W}&:\mathcal{K}_{2,W} \rightarrow H \\
T ^{*}_{W}\lbrace v_i(C^{*}&\pi_{W_{i}} C')^{\frac{1}{2}}f\rbrace_{i\in\Bbb I}=\sum _{i\in\Bbb I} v_{i}^{2}C^{*}\pi_{W_{i}} C'f,
\end{align*}
is a well-defined, bounded and surjective with $\Vert T^*_W\Vert\leq\sqrt{B}$.
\end{theorem}
\begin{proof}
The sufficient condition is clear. For the opposite case, assume that the operator $T^*_{W}$ is a well-defined, bounded and surjective with $\Vert T^*_W\Vert\leq\sqrt{B}$ and $\Bbb J\subset\Bbb I$ such that $\vert\Bbb J\vert<\infty$. For each $f\in H$ we have,
\begin{align*}
\sum _{i\in\Bbb J} v_{i}^{2} \Vert(C^* \pi_{W_{i}} C')^{\frac{1}{2}}f\Vert^2&=\big\langle\sum _{i\in\Bbb J} v_{i}^{2}C^{*}\pi_{W_{i}} C'f,f\big\rangle\\
&=\big\langle T^*_{W}\lbrace v_i(C^{*}\pi_{W_{i}} C')^{\frac{1}{2}}f\rbrace_{i\in\Bbb J},f\big\rangle\\
&\leq\Vert T^*_{W}\Vert\big\Vert\lbrace v_i(C^{*}\pi_{W_{i}} C')^{\frac{1}{2}}f\rbrace_{i\in\Bbb J}\big\Vert\Vert f\Vert.
\end{align*}
Hence
$$\sum _{i\in\Bbb I} v_{i}^{2} \Vert(C^* \pi_{W_{i}} C')^{\frac{1}{2}}f\Vert^2\leq\Vert T^*_W\Vert^2\Vert f\Vert^2,$$
and this means that $W$ is a $(C,C')$-controlled Bessel fusion sequence with bound $B$. Since $T^*_W$ is bounded and surjective, so there exists the pseudo-inverse operator $T^{*\dagger}_W$ which $T^*_WT^{*\dagger}_W f=f$ for each $f\in H$. Suppose that $f\in H$ and we can write,
\begin{align*}
\Vert f\Vert^4&=\vert\langle T^*_WT^{*\dagger}_W f,f\rangle\vert^2
=\vert\langle T^{*\dagger}_W f,T_Wf\rangle\vert^2\leq\Vert T^{*\dagger}_W\Vert^2\Vert f\Vert^2\Vert T_W f\Vert^2.
\end{align*}
So, we conclude that
$$\Vert T^{*\dagger}_W\Vert^{-2}\Vert f\Vert^2\leq\sum _{i\in\Bbb I} v_{i}^{2} \Vert(C^* \pi_{W_{i}} C')^{\frac{1}{2}}f\Vert^2.$$
\end{proof}
\begin{theorem}
Let $C^*\pi_{W_i}C'$ is a positive operator for each
$i\in\Bbb I$. A sequence  $W=\lbrace (W_{i},v_{i})\rbrace _ {i \in \Bbb I}$ is  a  $(C,C')$-controlled fusion Bessel sequence  for $H$ if and only if $\sum _{i\in\Bbb J} v_{i}^{2}C^{*}\pi_{W_{i}} C'f$ converges for all $\lbrace v_i(C^{*}\pi_{W_{i}} C')^{\frac{1}{2}}f\rbrace_{i\in\Bbb I}\in\mathcal{K}_{2,W}$.
\end{theorem}
\begin{proof}
Assume that $\sum _{i\in\Bbb J} v_{i}^{2}C^{*}\pi_{W_{i}} C'f$ converges for any $\lbrace v_i(C^{*}\pi_{W_{i}} C')^{\frac{1}{2}}f\rbrace_{i\in\Bbb I}\in\mathcal{K}_{2,W}$. Define
\begin{align*}
U_{W}&:\mathcal{K}_{2,W} \rightarrow H \\
U_{W}\lbrace v_i(C^{*}&\pi_{W_{i}} C')^{\frac{1}{2}}f\rbrace_{i\in\Bbb I}=\sum _{i\in\Bbb I} v_{i}^{2}C^{*}\pi_{W_{i}} C'f,
\end{align*}
and
\begin{align*}
U_{n,W}&:\mathcal{K}_{2,W} \rightarrow H \\
U_{n,W}\lbrace v_i(C^{*}&\pi_{W_{i}} C')^{\frac{1}{2}}f\rbrace_{i\in\Bbb I}=\sum _{i=1}^{n} v_{i}^{2}C^{*}\pi_{W_{i}} C'f,
\end{align*}
for each $n\in\Bbb N$. Thus, $U_W$ is  well-defined and also
$$\Vert U_{n,W}\lbrace v_i(C^{*}\pi_{W_{i}} C')^{\frac{1}{2}}f\rbrace_{i\in\Bbb I}\Vert\leq B_n\Vert\lbrace v_i(C^{*}\pi_{W_{i}} C')^{\frac{1}{2}}f\rbrace_{i\in\Bbb I}\Vert,$$
where $B_n^2=\sum _{i=1}^{n} v_{i}^{2} \Vert(C^* \pi_{W_{i}} C')^{\frac{1}{2}}\Vert^2$.
Therefore, $\lbrace U_{n,W}\rbrace$ is a sequence of bounded linear operators which converges pointwise to $U_W$. So, by the Banach-Steinhaus Theorem, $U_W$ is a bounded operator with
$$\Vert U_W\Vert\leq\liminf_{n}\Vert U_{n,W}\Vert.$$
Hence, with the same method of the proof of Theorem \ref{char1}, we get $W$ is a $(C,C')$-fusion Bessel sequence. The converse is evident.
\end{proof}
%%%%%%%%%%%%%%%%%%%%%%%%%%%%%%%%%
\section{Erasures of subspaces and the approximation operator}
First, we provide sufficient conditions on the weights for a subspace to be deleted yet still leave a controlled fusion frame with the same method in \cite{cas2}. Let  $\Bbb J\subset\Bbb I$ and for any $i\in\Bbb J$ we define
$$\mathcal{M}_i=\{f\in H : \ (C^{*}\pi_{W_{i}} C')^{\frac{1}{2}}f=f\}.$$
It is easy to check that $\mathcal{M}_i$ is a closed subspace of $H$ for each $i\in\Bbb J$. Indeed, if $\{f_n\}$ is a sequence in $\mathcal{M}_i$ such that $f_n\rightarrow f$ and $\varepsilon>0$, then
$$\Vert(C^{*}\pi_{W_{i}} C')^{\frac{1}{2}}f-f\Vert\leq\Vert(C^{*}\pi_{W_{i}} C')^{\frac{1}{2}}f-(C^{*}\pi_{W_{i}} C')^{\frac{1}{2}}f_n+f_n-f\Vert<\varepsilon.$$
\begin{theorem}\label{sharp}
Let $W$ be a $(C,C')$-controlled fusion frame with bounds $A$ and $B$ and also  $C^*\pi_{W_i}C'$ is a positive operator for each
$i\in\Bbb I$. Then the following statements hold.
\begin{enumerate}
\item[(I)] If $\sum_{i\in\Bbb J}v_i^2>B$, then $\bigcap_{i\in\Bbb J}\mathcal{M}_i=\{0\}$.
\item[(II)] If $\sum_{i\in\Bbb J}v_i^2=B$, then $\bigcap_{i\in\Bbb J}\mathcal{M}_i\subseteq\ker\{(C^* \pi_{W_{i}} C')^{\frac{1}{2}}\}_{i\in\Bbb I\setminus\Bbb J}$.
\item[(III)] If $\alpha:=\sum_{i\in\Bbb J}v_i^2<A$, then $\{(W_i, v_i)\}_{i\in\Bbb I\setminus\Bbb J}$ is a $(C,C')$-controlled fusion frame with bounds $A-\alpha$ and $B$.
\end{enumerate}
\end{theorem}
\begin{proof}
(I). Assume that $f\in\bigcap_{i\in\Bbb J}\mathcal{M}_i$, we have
\begin{align*}
B\Vert f\Vert^2&<\Big(\sum_{i\in\Bbb J}v_i^2\Big)\Vert f\Vert^2\\
&\leq\sum _{i\in\Bbb J} v_{i}^{2} \Vert(C^* \pi_{W_{i}} C')^{\frac{1}{2}}f\Vert^2+\sum _{i\in\Bbb I\setminus\Bbb J} v_{i}^{2} \Vert(C^* \pi_{W_{i}} C')^{\frac{1}{2}}f\Vert^2\\
&\leq B\Vert f\Vert^2.
\end{align*}
Hence $f=0$.

(II). Again, suppose that $f\in\bigcap_{i\in\Bbb J}\mathcal{M}_i$ and we compute that
\begin{align*}
B\Vert f\Vert^2&=\sum _{i\in\Bbb J} v_{i}^{2} \Vert(C^* \pi_{W_{i}} C')^{\frac{1}{2}}f\Vert^2\\
&\leq\sum _{i\in\Bbb J} v_{i}^{2} \Vert(C^* \pi_{W_{i}} C')^{\frac{1}{2}}f\Vert^2+\sum _{i\in\Bbb I\setminus\Bbb J} v_{i}^{2} \Vert(C^* \pi_{W_{i}} C')^{\frac{1}{2}}f\Vert^2\\
&\leq B\Vert f\Vert^2.
\end{align*}
Therefore $\sum _{i\in\Bbb I\setminus\Bbb J} v_{i}^{2} \Vert(C^* \pi_{W_{i}} C')^{\frac{1}{2}}f\Vert^2=0$ and this prove item (II).

(III). The upper bound is evident. Let $f\in H$ and we get
\begin{align*}
\sum _{i\in\Bbb I\setminus\Bbb J} v_{i}^{2} \Vert(C^* \pi_{W_{i}} C')^{\frac{1}{2}}f\Vert^2&=\sum _{i\in\Bbb I} v_{i}^{2} \Vert(C^* \pi_{W_{i}} C')^{\frac{1}{2}}f\Vert^2-\sum _{i\in\Bbb J} v_{i}^{2} \Vert(C^* \pi_{W_{i}} C')^{\frac{1}{2}}f\Vert^2\\
&\geq A\Vert f\Vert^2-\Big(\sum_{i\in\Bbb J}v_i\Big)\Vert f\Vert^2\\
&=(A-\alpha)\Vert f\Vert^2.
\end{align*}
\end{proof}
\begin{remark}
The claim in item (III) is sharp. For this, let $H=\Bbb R^3$ with the standard orthonormal basis $\{e_1, e_2, e_3\}$ and we define
$$W_1=\mathop{\rm span}\{e_1, e_2\},\quad  W_2=\mathop{\rm span}\{e_2, e_3\},\quad W_3=\mathop{\rm span}\{e_3\},$$
and $v_i=\frac{1}{\sqrt{2}}$ for each $i=1,2,3$.
Then, $\{(W_i,v_i)\}_{i=1,2,3}$ is a Parseval fusion frame. So, by Proposition 2.2 in \cite{Khosravi}, it is a $C^2$-controlled fusion frame for every $C\in\mathcal{GL}(H)$. So, by Theorem \ref{sharp} item (III), one subspace can be deleted yet leaving a controlled fusion frame. However, obviously the deletion of some pairs of subspaces destroys the controlled fusion frame property.
\end{remark}
The following corollary yields immediately when one single subspace can be deleted.
\begin{corollary}
Let $W$ be a $(C,C')$-controlled fusion frame with bounds $A$ and $B$.
If there exists $i_0\in\Bbb I$ such that $v_{i_0}^2<A$, then $\{(W_i, v_i)\}_{i\in\Bbb I,i\neq i_0}$ is a $(C,C')$-controlled fusion frame with bounds $A-v_{i_0}^2$ and $B$.
\end{corollary}
Next, we make the reconstruction error for controlled fusion frames. The optimality under erasures has been presented for frames in \cite{holmes} and for Parseval fusion frames have been studied in \cite{bod,cas2}.
 Suppose that $\Bbb I=\{1,2,\cdots, m\}$ is finite and $W$ is a Parseval $(C,C')$-controlled fusion frame for $H$ where $\dim H<\infty$. For every $i_0\in\Bbb I$, we consider the following operator:
\begin{align*}
\mathcal{D}&_{i_0}:\mathcal{K}_{2,W}\longrightarrow \mathcal{K}_{2,W},\\
\mathcal{D}_{i_0}\lbrace &v_i(C^{*}\pi_{W_{i}} C')^{\frac{1}{2}}f\rbrace_{i\in\Bbb I}=\delta_{i,i_0}v_{i_0}(C^{*}\pi_{W_{i_0}} C')^{\frac{1}{2}}f.
\end{align*}
We define the associated 1-erasure reconstruction error $\mathcal{E}_1(W)$ to be
$$\mathcal{E}_1(W)=\max_{i\in\Bbb I}\Vert T^*_{W}\mathcal{D}_i T_{W}\Vert.$$
Since
$$\Vert T^*_{W}\mathcal{D}_i T_{W}\Vert=\sup_{\Vert f\Vert=1}\Vert T^*_{W}\mathcal{D}_i T_{W}f\Vert=v_i^2\sup_{\Vert f\Vert=1}\Vert C^*\pi_{W_i}C'f\Vert\leq v_i^2\Vert C\Vert\Vert C'\Vert,$$
therefore,
$$\mathcal{E}_1(W)=\max_{i\in\Bbb I}v_i^2\Vert C\Vert\Vert C'\Vert.$$
\begin{theorem}
Let $W$ be a Parseval $(C,C')$-controlled fusion frame for $H$ where $\dim H=n$. Then the following conditions are equivalent.
\begin{enumerate}
\item[(I)] $W$ satisfies
$\mathcal{E}_1(W)=\min_{i\in\Bbb I}\mathcal{E}_1(\widetilde{W}_i, \widetilde{v}_i)_{i\in\Bbb I}$, where $(\widetilde{W}_i, \widetilde{v}_i)_{i\in\Bbb I}$ is a Parseval $(C,C')$-controlled fusion frame  for $H$ with $\dim\widetilde{W_i}=\dim W_j$ for each $i\in\Bbb I$.
\item[(II)] For each $i\in\Bbb I$ we have
$$v_i^2\Vert C\Vert\Vert C'\Vert=\frac{n}{m.\dim W_i}.$$
\end{enumerate}
\end{theorem}
\begin{proof}
Assume that $\{e_{ij}\}_{j\in\Bbb J_i}$ is an orthonormal basis for $W_i$ for each $i\in\Bbb I$. Via Theorem 2.10 in \cite{Khosravi}, the sequence
$\{v_i C'^* e_{ij}\}_{i=1, j=1}^{m, \quad \dim W_i}$
is a Parseval $C^*(C'^*)^{-1}$-controlled frame for $H$. By Corollary \ref{eigin2}, we can get
$$n=\sum_{i=1}^{m}\sum_{j=1}^{\dim W_i}v_i^2\langle C'^{-1}CC'^*e_{ij}, C'^*e_{ij}\rangle\leq\sum_{i=1}^{m}v_i^2\dim W_i \Vert C\Vert\Vert C'\Vert.$$
So, there exists $i\in\Bbb I$ such that
$$n\leq m.v_i^2\dim W_i \Vert C\Vert\Vert C'\Vert.$$
Since the dimensions as well as the number of subspaces are fixed, we conclude that $\mathcal{E}_1(\Lambda)$ is minimal if and only if
$$n=m. v_i^2\dim W_i \Vert C\Vert\Vert C'\Vert \ , \ \ \ \ \ (\forall i\in\Bbb I).$$
\end{proof}
Finally, we study the composition of  the controlled synthesis and analysis operators of two controlled fusion frames. This topic has been studied in \cite{fa3} for continuous frames and for g-fusion frames have been discussed by the author in \cite{sad1}.
\begin{theorem}
Let $\Bbb I$ be finite, $W=(W_i, v_i)$ and
$Z=(Z_i,  w_i)$ be two $(C,C')$-controlled fusion Bessel sequence for $H$, also $C^*\pi_{W_i}C'$ and $C^*\pi_{Z_i}C'$ are positive operators for each
$i\in\Bbb I$. If $\phi:=T^*_{W}T_{Z}$,
then $\phi$ is a trace class operator.
\end{theorem}
\begin{proof}
Suppose that $\phi=u\vert\phi\vert$ is a polar decomposition of the operator $\phi$, where
$u\in\mathcal{B}(H)$ is a partial isometry. Hence, $\vert\phi\vert=u^*T^*_{W}T_{Z}$. Assume that
$\{e_i\}_{i\in\Bbb I}$ is an orthonormal basis for $H$, then
\begin{align*}
tr(\vert\phi\vert)&=\sum_{i\in\Bbb I}\langle \vert\phi\vert e_i, e_i\rangle\\
&=\sum_{i\in\Bbb I}\langle T_{Z}e_i, T_{W}ue_j\rangle\\
&=\sum_{i\in\Bbb I}\big\langle\{w_j(C^*\pi_{Z_j}C')^{\frac{1}{2}}e_i\}_{j\in\Bbb I}, \{v_j(C^*\pi_{Z_j}C')^{\frac{1}{2}}ue_i\}_{j\in\Bbb I}\big\rangle\\
&=\sum_{i\in\Bbb I}\sum_{j\in\Bbb I}\langle w_j(C^*\pi_{Z_j}C')^{\frac{1}{2}}e_i, v_j(C^*\pi_{Z_j}C')^{\frac{1}{2}}ue_i\rangle\\
&\leq\sum_{i\in\Bbb I}\sum_{j\in\Bbb I}\Vert w_j(C^*\pi_{Z_j}C')^{\frac{1}{2}}e_i\Vert. \Vert v_j(C^*\pi_{Z_j}C')^{\frac{1}{2}}ue_i\Vert\\
&\leq\sum_{i\in\Bbb I}\Big(\sum_{j\in\Bbb I}\Vert w_j(C^*\pi_{Z_j}C')^{\frac{1}{2}}e_i\Vert^2\Big)^{\frac{1}{2}}\Big(\sum_{j\in\Bbb I}\Vert v_j(C^*\pi_{Z_j}C')^{\frac{1}{2}}ue_i\Vert^2\Big)^{\frac{1}{2}}\\
&\leq\sum_{i\in\Bbb I}\sqrt{B_{W}B_{Z}} \Vert ue_i\Vert\\
&\leq\sqrt{B_{W}B_{Z}} \ \vert\Bbb I\vert\Vert u\Vert<\infty.
\end{align*}
\end{proof}
Suppose that $W=\{W_i\}_{i\in\Bbb I}$ and $Z=\{Z_i\}_{i\in\Bbb I}$ are two sequences of closed subspaces of $H$ and $\{v_i\}_{i\in\Bbb I}$ is a set of weights. If $C^*\pi_{W_i}C'$ and $C^*\pi_{Z_i}C'$ are positive operators for each
$i\in\Bbb I$, we define the $CC'$-approximation operator
\begin{align*}
\Phi&: H\longrightarrow H,\\
\Phi f=&\sum_{i\in\Bbb I}v_i(C^*\pi_{Z_i}C')^{\frac{1}{2}}\big(v_i(C^*\pi_{W_i}C')^{\frac{1}{2}}f\big).
\end{align*}
\begin{theorem}
Let $A_1, A_2>0$ and $0\leq\gamma<1$ be real numbers such that for each  $f\in H$ and $\{f_i\}_{i\in\Bbb I}\in\mathcal{K}_{2,Z}$, the following assertions holds:
\begin{enumerate}
\item[(I)] $\sum_{i\in\Bbb I}v_i^2\Vert (C^*\pi_{W_i}C')^{\frac{1}{2}}f\Vert^2\leq A_1\Vert f\Vert^2$;
\item[(II)] $\Vert\sum_{i\in\Bbb I}v_i(C^*\pi_{Z_i}C')^{\frac{1}{2}}f_i\Vert^2\leq A_2\Vert\{f_i\}\Vert_2^2$;
\item[(III)] $\Vert f-\Phi f\Vert^2\leq\gamma\Vert f\Vert^2$.
\end{enumerate}
Then $\{(W_i,v_i)\}_{i\in\Bbb I}$ is a $(C,C')$-controlled fusion frame for $H$ with bounds
$A_{2}^{-1}(1-\gamma)^2$ and $A_1$. Also, $\{(Z_i,v_i)\}_{i\in\Bbb I}$ is a $(C,C')$-controlled fusion frame for $H$ with bounds
$A_{1}^{-1}(1-\gamma)^2$ and $A_2$.
\end{theorem}
\begin{proof}
Assume that $f\in H$, with items (I) and (II) we get
\begin{align*}
\Vert \Phi f\Vert^2\leq A_2\Vert\{v_i(C^*\pi_{W_i}C')^{\frac{1}{2}}f\}\Vert^2_2=A_2\sum_{i\in\Bbb I}v_i^2\Vert (C^*\pi_{W_i}C')^{\frac{1}{2}}f\Vert^2\leq A_1A_2\Vert f\Vert^2.
\end{align*}
Hence, $\Phi$ is a bounded operator. So, $\Phi$ is invertible and
$\Vert\Phi^{-1}\Vert\leq(1-\gamma)^{-1}$. Thus,
\begin{align*}
\Vert f\Vert^2&=\Vert\Phi^{-1}\Phi f\Vert^2\\
&\leq(1-\gamma)^{-2}\Vert\Phi f\Vert^2\\
&\leq A_2(1-\gamma)^{-2}\sum_{i\in\Bbb I}v_i^2\Vert (C^*\pi_{W_i}C')^{\frac{1}{2}}f\Vert^2\\
&\leq A_2 A_1(1-\gamma)^{-2}\Vert f\Vert^2.
\end{align*}
Hence, we conclude that
$$A_2^{-1}(1-\gamma)^2\Vert f\Vert^2\leq\sum_{i\in\Bbb I}v_i^2\Vert (C^*\pi_{W_i}C')^{\frac{1}{2}}f\Vert^2\leq A_1\Vert f\Vert^2,$$
and the first part is proved. Next, we verify two inequalities which are dual to (I) and (II) for $\{(Z_i,v_i)\}_{i\in\Bbb I}$. Let $f\in H$ and we have
\begin{align*}
\Big(\sum_{i\in\Bbb I}v_i^2\Vert (C^*\pi_{Z_i}C')^{\frac{1}{2}}f\Vert^2\Big)^2&=\Big(\big\langle\sum_{i\in\Bbb I}v_i^2C^*\pi_{Z_i}C'f, f\big\rangle\Big)^2\\
&\leq\Vert\sum_{i\in\Bbb I}v_i^2C^*\pi_{Z_i}C'f\Vert^2\Vert f\Vert^2\\
&\leq A_2\Vert f\Vert^2\sum_{i\in\Bbb I}v_i^2\Vert (C^*\pi_{Z_i}C')^{\frac{1}{2}}f\Vert^2.
\end{align*}
Therefore,
$$\sum_{i\in\Bbb I}v_i^2\Vert (C^*\pi_{Z_i}C')^{\frac{1}{2}}f\Vert^2\leq A_2\Vert f\Vert^2.$$
For second inequality, if
$\{f_i\}_{i\in\Bbb I}\in\mathcal{K}_{2,W}$, we can write
\begin{align*}
\Vert\sum_{i\in\Bbb I}v_i(C^*\pi_{W_i}C')^{\frac{1}{2}}f_i\Vert^2&=\Big(\sup_{\Vert f\Vert=1}\big\vert\big\langle\sum_{i\in\Bbb I}v_i(C^*\pi_{W_i}C')^{\frac{1}{2}}f_i, f\big\rangle\big\vert\Big)^2\\
&\leq\Big(\sup_{\Vert f\Vert=1}\big\vert\sum_{i\in\Bbb I}\big\langle f_i, v_i(C^*\pi_{W_i}C')^{\frac{1}{2}}f\big\rangle\big\vert\Big)^2\\
&\leq\Vert\{f_i\}\Vert^2_2\Big(\sup_{\Vert f\Vert=1}\sum_{i\in\Bbb I}v_i^2\Vert (C^*\pi_{W_i}C')^{\frac{1}{2}} f\Vert^2\Big)\\
&\leq A_1\Vert\{f_j\}\Vert^2_2.
\end{align*}
Now by similar argument and applying an $CC'$-approximation operator of the form
$$\Phi^* f=\sum_{i\in\Bbb I}v_i(C^*\pi_{W_i}C')^{\frac{1}{2}}\big(v_i(C^*\pi_{Z_i}C')^{\frac{1}{2}}f\big),$$
we can establish $\{(Z_i,v_i)\}_{i\in\Bbb I}$ has required properties.
\end{proof}

\end{document}